\documentclass[12pt]{article}
\usepackage{mathrsfs}
\usepackage{amsmath}
\usepackage{amsmath,amsthm,amssymb,amscd}
\usepackage{latexsym}
\usepackage[colorlinks,linkcolor=blue,anchorcolor=blue,citecolor=blue,CJKbookmarks=True]{hyperref}
\usepackage[numbers,sort&compress]{natbib}
\usepackage{cases}
\usepackage{makecell}
\usepackage{multirow}
\usepackage{geometry}
\usepackage{tabularx}
\usepackage{graphicx}
\usepackage{tikz}
\usepackage{tikz-cd}

\allowdisplaybreaks

\newtheorem{theorem}{Theorem}[section]
\newtheorem{corollary}[theorem]{Corollary}
\newtheorem{lemma}[theorem]{Lemma}
\newtheorem{example}[theorem]{Example}
\newtheorem{proposition}[theorem]{Proposition}
\newtheorem{question}[theorem]{Question}

\theoremstyle{definition}
\newtheorem{definition}[theorem]{Definition}
\newtheorem{remark}[theorem]{Remark}

\numberwithin{equation}{section}

\begin{document}

\begin{center}
{\large  \bf Transposed Triple Products and Pro-Symmetric Rings in $\ast$-Rings}\\
\vspace{0.8cm}   Huaxi Chen$^{a}$,  Long Wang$^{b}$and Honglin Zou$^{c}$\footnote{Corresponding author. Email: honglinzou@sina.com or lwangmath@yzu.edu.cn} \\
\vspace{0.5cm} {\small $^{a}$School of Mathematics and Physics, Bengbu University, Bengbu, China\\
\small $^{b}$School of Mathematical Sciences, Yangzhou University, Yangzhou,  China\\
\small $^{c}$College of Basic Science, Zhejiang Shuren University, Hangzhou, P. R. China}
\end{center}

\bigskip

{ \bf  Abstract:}  \leftskip0truemm\rightskip0truemm
This paper investigates properties concerning transposed triple products in rings.
Motivated by Cline's formula, we characterize symmetric rings by means of group invertible elements and EP elements.
We prove that a unital ring $R$ is symmetric if and only if $abc\in R^{\sharp}$ implies $acb\in R^{\sharp}$ for all $a,b,c\in R$.
In particular, we give an answer to the problem posed in \cite[Problem 2.9]{MW1}.
An example is provided to illustrate that for a symmetric ring $R$, $abc=e$ does not generally yield $acb=e$.
For $\ast$-rings, we introduce the notion of pro-symmetric rings:
a ring $R$ is pro-symmetric if $abc\in P(R)$ implies $acb\in P(R)$ for all $a,b,c\in R$.
We show that every pro-symmetric ring is symmetric.
Several counterexamples are constructed to distinguish these classes of rings,
and their mutual inclusion relations are also discussed.
 \bigskip

\section{\normalsize Introduction}

Generalized inverse theory is an important subject within non-commutative ring theory,
which has found extensive applications in matrix theory, operator algebras and many other related areas.
Among various generalized inverses, the Drazin inverse and the Moore-Penrose inverse
are two fundamental and widely studied concepts.
As a classical result concerning Drazin invertibility, Cline's formula establishes an intimate connection
between the generalized invertibility of the products $ab$ and $ba$ over arbitrary rings \cite{C1}.
That is, if $ab$ is Drazin invertible, then $ba$ is also Drazin invertible and $(ba)^{D}=b\bigl((ab)^{D}\bigr)^{2}a$.

Much of the existing literature focuses on element-level problems \cite{LZ1,LCC1,MA1,MZ1,Mo1,SCLZ1}.
Given two ring elements $a,b$, researchers seek necessary and sufficient conditions under which
the reversed product $ba$ inherits generalized invertibility from $ab$.
These studies treat such inheritance as a local feature of individual elements.

In contrast, this paper shifts perspective from local element-wise behaviour
to global structural properties of rings. Rather than limiting our scope to two-factor reversed products,
we exploit properties of group-invertible and EP elements for triple reversed products
to characterize several ring classes.

Symmetric and reversible rings are basic ring families defined by zero-product conditions \cite{Co1,M1}.
They are closely related to annihilators, idempotents, abelian rings and Dedekind finiteness \cite{AC1,FN1,KL1}.
Within the $\ast$-ring framework, these notions interact naturally with projections and the Moore-Penrose inverse.
In \cite{Wa1}, the author proved that $R$ is Dedekind finite if and only if $ab\in U(R)$ implies $ba\in U(R)$.
By means of generalized inverses, it has been established that $R$ is reversible
if and only if $ab\in R^{\sharp}$ implies $ba\in R^{\sharp}$, which is further equivalent to $ab\in R^{\mathrm{EP}}$ implies $ba\in R^{\mathrm{EP}}$.

Inspired by Cline-type transfer phenomena for reversed-product pairs and existing characterizations of reversible rings,
we derive new equivalent descriptions for symmetric rings.
We prove that a unital ring $R$ is symmetric if and only if group invertibility transfers from $abc$ to $acb$ for all $a,b,c\in R$.
For $\ast$-rings, $R$ is symmetric precisely when the EP property transfers from $abc$ to $acb$.
In particular, we resolve the problem posed in \cite[Problem 2.9]{MW1}.
An example shows that for a symmetric ring $R$, $abc=e$ does not generally force $acb=e$.

Using projections, we introduce the class of pro-symmetric rings, which forms a subclass of symmetric rings.
Several counterexamples are constructed to separate these ring classes, and their mutual inclusion relations are investigated.

\section{\normalsize Preliminaries}

Throughout this paper, all rings are associative with unity.
In this section, we collect some basic definitions and notation needed in the sequel.

An element $a\in R$ is Drazin invertible \cite{Dr1} if there exists an element $x\in R$ satisfying
\[
a^{k}xa=a^{k},\quad xax=x,\quad ax=xa
\]
for some positive integer $k$. The smallest such integer $k$ is known as the Drazin index of $a$,
and the unique element $x$ is called the Drazin inverse of $a$, written as $a^{D}$.
In particular, $a$ is group invertible precisely when its Drazin index is equal to $1$.
The group inverse is unique and is denoted by $a^{\sharp}$.
We let $R^{\sharp}$ stand for the set of all group-invertible elements of $R$.

An involution on a ring $R$ is an order-$2$ anti-isomorphism $a\mapsto a^{\ast}$ such that
\[
(a^{\ast})^{\ast}=a,\quad (a+b)^{\ast}=a^{\ast}+b^{\ast},\quad (ab)^{\ast}=b^{\ast}a^{\ast}.
\]
A ring endowed with an involution is referred to as a $\ast$-ring.

Given a $\ast$-ring $R$, an element $a\in R$ is Moore-Penrose invertible \cite{PT1} if there exists $x\in R$ with
\[
axa=a,\quad xax=x,\quad (ax)^{\ast}=ax,\quad (xa)^{\ast}=xa.
\]
The unique element $x$ is the Moore-Penrose inverse of $a$, denoted by $a^{\dag}$.
We use $R^{\dag}$ to denote the set of all Moore-Penrose invertible elements of $R$.

Let $R$ be a $\ast$-ring. An element $a\in R$ is said to be an EP element
if $a\in R^{\sharp}\cap R^{\dag}$ and $a^{\sharp}=a^{\dag}$.
The collection of all EP elements of $R$ is denoted by $R^{\mathrm{EP}}$.
An element $p\in R$ is a projection provided that $p^{2}=p=p^{\ast}$.
We write $E(R)$ and $P(R)$ for the sets of all idempotents and all projections of $R$, respectively.
The center of $R$ is denoted by $C(R)$.

\medskip
We next recall several well-known classes of rings defined by zero-product conditions,
which play essential roles throughout our investigation.

A ring $R$ is {\it reversible} if $ab=0$ implies $ba=0$ for all $a,b\in R$.

A ring $R$ is {\it symmetric} if $abc=0$ implies $acb=0$ for all $a,b,c\in R$.

A ring $R$ is {\it abelian} if $E(R)\subseteq C(R)$.

\section{\normalsize Characterizations of Symmetric Rings by Group Invertibility}

In this section, we investigate characterizations of symmetric rings by means of group-invertible elements and EP elements.
We begin with the fundamental properties of reversible rings.

\begin{lemma}\label{lem2.01}
{\rm(\cite{Wa1}, Theorem 3.4)}
Let $R$ be a ring with unity, and let $a,b\in R$. The following statements are equivalent:

(1) $R$ is reversible.

(2) $ab\in R^{\sharp}$ implies $ba\in R^{\sharp}$.
\end{lemma}

\begin{proposition}\label{pro2.02}
Let $R$ be a unital ring. The following statements are equivalent:

(1) $R$ is symmetric.

(2) $abc=0$ implies $acb\in E(R)$ for all $a,b,c\in R$.

(3) $abc=0$ implies $acb\in R^{\sharp}$ for all $a,b,c\in R$.
\end{proposition}

\begin{proof}
$(1)\Rightarrow(2)$ and $(2)\Rightarrow(3)$.
Trivially.

$(3)\Rightarrow(1)$.
We first show that $R$ is reversible.
Take $x,y\in R$ such that $xy=0$. Then condition (3) gives $yx\in R^{\sharp}$.
Then
\[
yx=(yx)(yx)(yx)^{\sharp}=0,
\]
which proves $R$ reversible.

Now assume $abc=0$. By condition (3), $acb\in R^{\sharp}$.
From Lemma \ref{lem2.01}, $bac\in R^{\sharp}$ and $cba\in R^{\sharp}$.
Moreover, note that reversible rings are abelian. Then $E(R) \subseteq C(R)$. Hence,
\[
\begin{aligned}
acb&=\bigl((acb)^{\sharp}\bigr)^{3}\,acbacbacbacb=\bigl((acb)^{\sharp}\bigr)^{3}\,ac\,(bac)\,ba\,(cba)\,cb \\
&=\bigl((acb)^{\sharp}\bigr)^{3}\,ac\,(bac)\bigl[(bac)(bac)^{\sharp}\bigr]\,ba\,(cba)\bigl[(cba)(cba)^{\sharp}\bigr]\,cb \\
&=\bigl((acb)^{\sharp}\bigr)^{3}\,a\bigl[(bac)(bac)^{\sharp}\bigr]c\,(bac)\,ba\,(cba)\bigl[(cba)(cba)^{\sharp}\bigr]\,cb \\
&=\bigl((acb)^{\sharp}\bigr)^{3}\,a\Bigl(b\bigl[(cba)(cba)^{\sharp}\bigr]ac\Bigr)(bac)^{\sharp}\,c\,(bac)\,ba\,(cba)\,cb \\
&=\bigl((acb)^{\sharp}\bigr)^{3}\,(abc)\,ba(cba)^{\sharp}ac(bac)^{\sharp}\,c\,(bac)\,ba\,(cba)\,cb =0.
\end{aligned}
\]
Consequently, $abc=0$ implies $acb=0$, and $R$ is symmetric.
\end{proof}

It is well known that an element $a$ in a unital ring $R$ is group invertible
if and only if $a\in a^{2}R\cap Ra^{2}$.

The next result constitutes one of the main theorems of this section,
which gives a new characterization of symmetric rings in terms of group invertibility of triple products.

\begin{theorem}\label{thm2.03}
Let $R$ be a unital ring. The following statements are equivalent:

(1) $R$ is symmetric.

(2) $abc \in R^{\sharp}$ implies $acb\in R^{\sharp}$ for all $a,b,c\in R$.
\end{theorem}

\begin{proof}
$(1)\Rightarrow(2)$. Suppose $abc \in R^{\sharp}$. By Lemma \ref{lem2.01},
we have $bca \in R^{\sharp}$ and $cab \in R^{\sharp}$.
Since $abc \in R^{\sharp}$, we have $(1-abc(abc)^{\sharp})abc=0$.
As $R$ is symmetric, this yields $(1-abc(abc)^{\sharp})acb=0$, i.e., $acb=abc(abc)^{\sharp}acb$.
Moreover, symmetric rings are abelian, so $acb=acb\,abc(abc)^{\sharp}$.

Observe that
\[
\begin{aligned}
abc(abc)^{\sharp}&=ab(cab)c(abc)^{\sharp}(abc)^{\sharp}\\
&=ab[cab(cab)^{\sharp}]cabc(abc)^{\sharp}(abc)^{\sharp}\\
&=a[cab(cab)^{\sharp}](bca)bc(abc)^{\sharp}(abc)^{\sharp}\\
&=a[cab(cab)^{\sharp}][bca(bca)^{\sharp}]bcabc(abc)^{\sharp}(abc)^{\sharp}\\
&=ac[bca(bca)^{\sharp}]ab(cab)^{\sharp}bcabc(abc)^{\sharp}(abc)^{\sharp}\\
&=acbca(bca)^{\sharp}ab(cab)^{\sharp}bc(abc)^{\sharp}.
\end{aligned}
\]
Therefore,
\begin{equation}\label{eq:2-1}
acb=acb\,abc(abc)^{\sharp}=(acb)^{2}ca(bca)^{\sharp}ab(cab)^{\sharp}bc(abc)^{\sharp}\in (acb)^{2}R.
\end{equation}
Similarly,
\[
\begin{aligned}
abc(abc)^{\sharp}&=(abc)^{\sharp}(abc)^{\sharp}ab(cab)c\\
&=(abc)^{\sharp}(abc)^{\sharp}abcab\bigl[(cab)^{\sharp}cab\bigr]c\\
&=(abc)^{\sharp}\bigl[(abc)^{\sharp}abc\bigr]abc\bigl[(cab)^{\sharp}cab\bigr]\\
&=(abc)^{\sharp}abc(cab)^{\sharp}cab\bigl[(abc)^{\sharp}abc\bigr]b\\
&=(abc)^{\sharp}abc(cab)^{\sharp}(cab)^{\sharp}cabca\bigl[(abc)^{\sharp}abc\bigr]b\\
&=(abc)^{\sharp}abc(cab)^{\sharp}(cab)^{\sharp}cabca\bigl[(bca)^{\sharp}bca\bigr]\bigl[(abc)^{\sharp}abc\bigr]b\\
&=(abc)^{\sharp}abc(cab)^{\sharp}(cab)^{\sharp}cabca(abc)^{\sharp}ab\bigl[(bca)^{\sharp}bca\bigr]cb.
\end{aligned}
\]
Hence,
\begin{equation}\label{eq:2-2}
acb=(abc)^{\sharp}abc(cab)^{\sharp}(cab)^{\sharp}cabca(abc)^{\sharp}ab\bigl[(bca)^{\sharp}bca\bigr]cb\,acb\in R(acb)^{2}.
\end{equation}
It then follows from \eqref{eq:2-1} and \eqref{eq:2-2} that $acb\in (acb)^{2}R \cap R(acb)^{2}$, so $acb\in R^{\sharp}$.

$(2)\Rightarrow(1)$.
If $abc=0$, then condition (2) gives $acb\in R^{\sharp}$.
Applying Proposition \ref{pro2.02}, we conclude that $R$ is symmetric.
\end{proof}

Recall that a ring $R$ is strongly regular if and only if every element of $R$ is group invertible.
The following corollary is an immediate consequence of Theorem \ref{thm2.03}.

\begin{corollary}\label{cor2.04}
Every strongly regular ring is symmetric.
\end{corollary}

\begin{lemma}\label{lem2.05}
If $R$ is symmetric and $abc\in R^{\sharp}$, then $abc(abc)^{\sharp}=acb(acb)^{\sharp}$.
\end{lemma}

\begin{proof}
From the proof of Theorem \ref{thm2.03}, we have $acb=abc(abc)^{\sharp}acb$, which yields
\[
acb(acb)^{\sharp}=abc(abc)^{\sharp}acb(acb)^{\sharp}.
\]
Observe that
\[
\begin{aligned}
abc(abc)^{\sharp}acb(acb)^{\sharp}
&=abcabc(abc)^{\sharp}(abc)^{\sharp}\bigl[acb(acb)^{\sharp}\bigr]\\
&=ab\bigl[cab(cab)^{\sharp}\bigr]cabc(abc)^{\sharp}(abc)^{\sharp}\bigl[acb(acb)^{\sharp}\bigr]\\
&=a\bigl[cab(cab)^{\sharp}\bigr]bcabc(abc)^{\sharp}(abc)^{\sharp}\bigl[acb(acb)^{\sharp}\bigr]\\
&=acab(cab)^{\sharp}\bigl[bca(bca)^{\sharp}\bigr]bcabc(abc)^{\sharp}(abc)^{\sharp}\bigl[acb(acb)^{\sharp}\bigr]\\
&=ac\bigl[bca(bca)^{\sharp}\bigr]ab(cab)^{\sharp}bcabc(abc)^{\sharp}(abc)^{\sharp}\bigl[acb(acb)^{\sharp}\bigr]\\
&=\bigl[acb(acb)^{\sharp}\bigr]acbca(bca)^{\sharp}ab(cab)^{\sharp}bcabc(abc)^{\sharp}(abc)^{\sharp}\\
&=ac\bigl[bca(bca)^{\sharp}\bigr]ab(cab)^{\sharp}bcabc(abc)^{\sharp}(abc)^{\sharp}\\
&=acab(cab)^{\sharp}\bigl[bca(bca)^{\sharp}\bigr]bcabc(abc)^{\sharp}(abc)^{\sharp}\\
&=a\bigl[cab(cab)^{\sharp}\bigr]bc(abc)^{\sharp}\\
&=abc(abc)^{\sharp}\bigl[cab(cab)^{\sharp}\bigr].
\end{aligned}
\]
Furthermore, $abc(abc)^{\sharp}\bigl[cab(cab)^{\sharp}\bigr]=abc(abc)^{\sharp}$.
This equality follows from the symmetry of $R$: since $(1-cab(cab)^{\sharp})cab=0$,
we have $(1-cab(cab)^{\sharp})abc=0$.
Therefore
\[
acb(acb)^{\sharp}=abc(abc)^{\sharp}acb(acb)^{\sharp}
=abc(abc)^{\sharp}\bigl[cab(cab)^{\sharp}\bigr]
=abc(abc)^{\sharp},
\]
which completes the proof.
\end{proof}

Exploiting the identity established in Lemma \ref{lem2.05},
we obtain an EP-element-based characterization of symmetric rings  within the setting of $\ast$-rings,
as stated in the following theorem.

\begin{theorem}\label{thm2.07}
Let $R$ be a $\ast$-ring. The following statements are equivalent:

(1) $R$ is symmetric.

(2) $abc \in R^{\mathrm{EP}}$ implies $acb\in R^{\mathrm{EP}}$ for all $a,b,c\in R$.
\end{theorem}

\begin{proof}
$(1)\Rightarrow(2)$. Suppose that $abc\in R^{\mathrm{EP}}$.
By Theorem \ref{thm2.03}, we obtain $acb\in R^{\sharp}$.
From Lemma \ref{lem2.05},
\[
acb(acb)^{\sharp}=abc(abc)^{\sharp}.
\]
Since $abc\in R^{\mathrm{EP}}$, we have $abc(abc)^{\sharp}=abc(abc)^{\dag}$.
Therefore $acb(acb)^{\sharp}=abc(abc)^{\dag}$, which implies that $acb\in R^{\mathrm{EP}}$.

$(2)\Rightarrow(1)$.
If $abc=0$, then $abc\in R^{\mathrm{EP}}$.
Condition $(2)$ yields $acb\in R^{\mathrm{EP}}$, and in particular $acb\in R^{\sharp}$.
Applying Proposition \ref{pro2.02}, we conclude that $R$ is symmetric.
\end{proof}

\begin{remark}
We emphasize that the hypotheses of Theorem \ref{thm2.03} and Theorem \ref{thm2.07} do not force the equality $abc=acb$.
Indeed, these theorems merely guarantee that group invertibility (resp., the EP property) of the product $abc$ is inherited by $acb$.
A concrete counterexample for this point will be provided in Remark \ref{rem2.09} using the real quaternion division ring $\mathbb{H}$.
\end{remark}

\begin{proposition}\label{pro2.08}
Let $R$ be a unital ring. The following statements are equivalent:

(1) $R$ is symmetric.

(2) $abc \in E(R)$ implies $acb\in R^{\sharp}$ for all $a,b,c\in R$.
\end{proposition}

\begin{proof}
$(1)\Rightarrow(2)$. This follows directly from Theorem \ref{thm2.03}.

$(2)\Rightarrow(1)$. If $abc=0$, then $abc\in E(R)$. By condition $(2)$, we obtain $acb\in R^{\sharp}$.
The conclusion now follows from the implication $(3)\Rightarrow(1)$ of Proposition \ref{pro2.02}.
\end{proof}

\begin{remark}\label{rem2.09}
Herein, it is natural to introduce the following two conditions:
\begin{align*}
& \text{Condition 1}: abc\in R^{\sharp} \text{ implies } acb\in E(R) \text{ for all }a,b,c\in R;\\
& \text{Condition 2}: abc\in E(R) \text{ implies } acb\in E(R) \text{ for all }a,b,c\in R.
\end{align*}
We now prove the implications
\[
\text{Condition 1}\implies \text{Condition 2}\implies R\text{ is symmetric}.
\]

\noindent\textbf{Condition 1} $\Rightarrow$ \textbf{Condition 2}:
If $abc\in E(R)$, then $abc\in R^{\sharp}$ since $E(R)\subseteq R^{\sharp}$.
By Condition 1, we obtain $acb\in E(R)$.

\noindent\textbf{Condition 2} $\Rightarrow$ $R$ is symmetric:
Suppose $abc=0$. Since $0\in E(R)$, Condition 2 yields $acb\in E(R)$.
Combined with the equivalence $(1)\Leftrightarrow(2)$ in Proposition \ref{pro2.02},
we conclude that $R$ is symmetric.

We next show that the reverse implications are not valid in general.

Firstly, Condition 2 does not imply Condition 1.
Take $R=\mathbb{Z}_{3}$.
As $R$ is commutative, Condition 2 is automatically satisfied.
Notice that $[2]\in R^{\sharp}$ while $[2]$ is not an idempotent element.
Set $a=[1],\,b=[1],\,c=[2]$. Then $abc=[2]\in R^{\sharp}$,
but $acb=[2]\notin E(R)$.
Hence Condition 1 fails.

Secondly, the symmetry of $R$ does not imply Condition 2.
Let $R=\mathbb{H}$ denote the real quaternion division ring, whose elements are of the form
\[
q = a + b\mathbf{i} + c\mathbf{j} + d\mathbf{k},\quad a,b,c,d\in\mathbb{R},
\]
with multiplication rules
\begin{align*}
& \mathbf{i}^2=\mathbf{j}^2=\mathbf{k}^2=\mathbf{i}\mathbf{j}\mathbf{k}=-1,\quad
\mathbf{i}\mathbf{j}=\mathbf{k},\;\mathbf{j}\mathbf{k}=\mathbf{i},\;\mathbf{k}\mathbf{i}=\mathbf{j},\\
& \mathbf{j}\mathbf{i}=-\mathbf{k},\;\mathbf{k}\mathbf{j}=-\mathbf{i},\;\mathbf{i}\mathbf{k}=-\mathbf{j}.
\end{align*}
$\mathbb{H}$ is a non-commutative division ring (skew-field) over $\mathbb{R}$.
Every nonzero element of $\mathbb{H}$ is invertible, so $\mathbb{H}$ is strongly regular.
By Corollary \ref{cor2.04}, $\mathbb{H}$ is symmetric.
Now choose $a=\mathbf{i},\;b=\mathbf{j},\;c=-\mathbf{k}$.
A direct computation yields $abc=1\in E(R)$, whereas $acb=-1\notin E(R)$.
\end{remark}

\begin{remark}\label{rem2.10}
If $R$ satisfies Condition 2, i.e.,
\[
abc\in E(R)\text{ implies }acb\in E(R)\text{ for all }a,b,c\in R,
\]
then $abc=acb$.

Indeed, $R$ is symmetric under this hypothesis.
From the argument given in Lemma \ref{lem2.05}, we obtain
\[
abc\,(abc)^{\sharp}=acb\,(acb)^{\sharp}.
\]
Since $abc,acb\in E(R)$, the group inverse of any idempotent element coincides with itself, namely $(abc)^{\sharp}=abc$ and $(acb)^{\sharp}=acb$.
Consequently, $abc=acb$.
\end{remark}

\begin{remark}\label{rem2.11}
In \cite[Problem 2.9]{MW1}, the authors raise the following problem:
\[
\text{Let }R\text{ be a symmetric ring and }e\in E(R).\text{ If }e=abc,\text{ does }e=acb\text{ hold?}
\]
In view of Remark \ref{rem2.09} and Remark \ref{rem2.10},
we obtain a negative answer to Problem 2.9 from \cite{MW1}.
\end{remark}

\begin{example}
(1) Let $R=\mathbb{H}$ denote the real quaternion division ring.
Then $R$ is strongly regular, yet it satisfies neither Condition 1 nor Condition 2.

(2) Let $R=\mathbb{Z}$ denote the ring of integers.
Then $R$ satisfies Condition 2, but $R$ is not strongly regular.

(3) Let $R=\mathbb{F}_2[x]$ be the polynomial ring over the binary field $\mathbb{F}_2$ in one indeterminate $x$.
Clearly, $R$ is a commutative ring.
An element of $\mathbb{F}_2[x]$ is group invertible if and only if it is a constant polynomial.
Hence $R^{\sharp}=\{0,1\}$, and every element in this set is idempotent.
Thus $R$ satisfies Condition 1, but $R$ fails to be strongly regular.
For instance, the indeterminate $x$ possesses no group inverse.
\end{example}

From Remark \ref{rem2.09}, we observe that
if $R$ is symmetric and $abc\in E(R)$, then $abc \neq acb$ in general.
Similarly, even for a symmetric ring $R$, the equality $abc=acb$ need not hold whenever $abc\in R^{\sharp}$.
In what follows, we establish several results concerning condition 2 from Remark \ref{rem2.09}.

\begin{proposition}\label{pro2.13}
Let $R$ be a unital ring. The following statements are equivalent:

(1) $abc\in E(R)$ implies $acb=abc$ for all $a,b,c\in R$.

(2) $abc\in E(R)$ implies $acb \in E(R)$ for all $a,b,c\in R$.

(3) $abc\in R^{\sharp}$ implies $acb=abc$ for all $a,b,c\in R$.
\end{proposition}

\begin{proof}
It is clear that $(1)\Leftrightarrow(2)$.
Indeed, the implication $(2)\Rightarrow(1)$ follows from Remark \ref{rem2.10}.

$(1)\text{ or }(2)\Rightarrow(3)$.
First, it follows from Remark \ref{rem2.09} that $R$ is symmetric.
Assume that $abc\in R^{\sharp}$. By Theorem \ref{thm2.03}, we have $acb\in R^{\sharp}$.
Moreover, Lemma \ref{lem2.05} yields $(abc)^{\sharp}abc=(acb)^{\sharp}acb$.
Since $(abc)^{\sharp}abc\in E(R)$, condition (1) implies $(abc)^{\sharp}abc=(abc)^{\sharp}acb$.
Hence $(abc)^{\sharp}acb=(acb)^{\sharp}acb$.
Multiplying both sides on the left by $abc$, we obtain
\[
abc(abc)^{\sharp}acb=abc(acb)^{\sharp}acb.
\]
We now verify the two equalities below.
Because $(1-acb(acb)^{\sharp})acb=0$ and $R$ is symmetric,
we get $(1-acb(acb)^{\sharp})abc=0$, which implies $abc(acb)^{\sharp}acb=abc$.
Similarly, $abc(abc)^{\sharp}acb=acb$.
Combining these identities, we conclude $abc=acb$.

$(3)\Rightarrow(1)$. If $abc\in E(R)$, then $abc\in R^{\sharp}$.
Condition (3) now forces $acb=abc$, and hence $acb\in E(R)$.
\end{proof}

\section{\normalsize Characterizations of Symmetric Rings by projections}

In Theorem \ref{thm2.07}, we have investigated characterizations of symmetric rings in terms of EP elements.
As is well known, this approach requires $R$ to be a $\ast$-ring.
This section further explores symmetric rings within $\ast$-rings,
establishing new characterizations by means of Moore-Penrose invertibility and projections.

\begin{proposition}\label{pro3.01}
Let $R$ be a $\ast$-ring. The following statements are equivalent:

(1) $R$ is symmetric.

(2) $abc=0$ implies $acb\in P(R)$ for all $a,b,c\in R$.

(3) $abc=0$ implies $acb \in R^{\mathrm{EP}}$ for all $a,b,c\in R$.
\end{proposition}

\begin{proof}
The implications $(1)\Rightarrow(2)$ and $(2)\Rightarrow(3)$ are obvious.
The converse $(3)\Rightarrow(1)$ follows from Proposition \ref{pro2.02}.
\end{proof}

Motivated by Theorem \ref{thm2.07} and Proposition \ref{pro3.01},
we introduce the following conditions concerning symmetric rings.

Let $R$ be a $\ast$-ring, and let $a,b,c\in R$. Consider the following statements:
\begin{enumerate}
    \item[(1)] $R$ is symmetric.
    \item[(2)] $abc \in P(R)$ implies $acb\in P(R)$ for all $a,b,c\in R$.
    \item[(3)] $abc=0$ implies $acb\in R^{\dag}$ for all $a,b,c\in R$.
    \item[(4)] $abc \in R^{\dag}$ implies $acb \in R^{\dag}$ for all $a,b,c\in R$.
\end{enumerate}

$(2)\Rightarrow(1)$. If $abc=0$, then $abc\in P(R)$, and hence $acb\in P(R)$.
By Proposition \ref{pro3.01}, we conclude that $R$ is symmetric.

Condition (1) does not imply Condition (2), as illustrated in Remark \ref{rem2.09}.
A suitable counterexample is given by the real quaternion division ring $\mathbb{H}$.
Note that an involution $\ast$ has not yet been specified on $\mathbb{H}$.
Nevertheless, for the chosen elements we have $abc=1$, which is a projection,
whereas $acb=-1$ is not idempotent and therefore cannot be a projection.

$(1)\Rightarrow(4)\Rightarrow (3)$.
First, we prove that if $R$ is symmetric, then every MP invertible element of $R$ is an EP element.
Let $x\in R^{\dag}$. Then there exists $x^{\dag}$ such that $xx^{\dag}x=x$.
We have $(1-xx^{\dag})x=0$ and $x(1-x^{\dag}x)=0$.
Since $R$ is symmetric, $(1-xx^{\dag})x=0$ yields $x(1-xx^{\dag})=0$.
Similarly, $x(1-x^{\dag}x)=0$ implies $(1-x^{\dag}x)x=0$.
It follows that $x=x^{2}x^{\dag}=x^{\dag}x^{2}$.
Therefore,
\[
x^{\dag}x=x^{\dag}x^{2}x^{\dag}=xx^{\dag},
\]
which shows $R^{\dag}=R^{\mathrm{EP}}$.

We now complete the proof of $(1)\Rightarrow(4)$.
If $abc \in R^{\dag}$, then $abc\in R^{\mathrm{EP}}$.
By Theorem \ref{thm2.07}, $acb\in R^{\mathrm{EP}}$, and hence $acb\in R^{\dag}$.
$(4)\Rightarrow(3)$. This is obvious.

$(1)\Rightarrow(4)\Rightarrow (3)$.
First, we claim that if $R$ is symmetric, then every MP-invertible element of $R$ is an EP element.
Let $x\in R^{\dag}$. Then there exists $x^{\dag}$ satisfying $xx^{\dag}x=x$.
We obtain $(1-xx^{\dag})x=0$ and $x(1-x^{\dag}x)=0$.
Since $R$ is symmetric, $(1-xx^{\dag})x=0$ yields $x(1-xx^{\dag})=0$.
Similarly, $x(1-x^{\dag}x)=0$ implies $(1-x^{\dag}x)x=0$.
It follows that $x=x^{2}x^{\dag}=x^{\dag}x^{2}$.
Therefore,
\[
x^{\dag}x=x^{\dag}x^{2}x^{\dag}=xx^{\dag},
\]
which shows $R^{\dag}=R^{\mathrm{EP}}$.

We now finish the proof of $(1)\Rightarrow(4)$.
If $abc \in R^{\dag}$, then $abc\in R^{\mathrm{EP}}$.
By Theorem \ref{thm2.07}, $acb\in R^{\mathrm{EP}}$, and hence $acb\in R^{\dag}$.

$(4)\Rightarrow(3)$. This is clear.

The example below demonstrates that the converses of $(1)\Rightarrow(3)$ and $(1)\Rightarrow(4)$ do not hold in general.

\begin{example}\label{exm3.02}
Let $R = M_2(\mathbb{C})$ be equipped
with the involution $\ast$ defined by conjugate transpose. Since $M_2(\mathbb{C})$ is $\ast$-regular,
the conditions $abc=0 \implies acb \in R^{\dag}$ and $abc\in R^{\dag} \implies acb \in R^{\dag}$ hold trivially.
Take
\[
a=
\begin{pmatrix}
1 & 1\\
0 & 0
\end{pmatrix},\quad
b=
\begin{pmatrix}
0 & 1\\
0 & 0
\end{pmatrix},\quad
c=
\begin{pmatrix}
1 & 0\\
0 & 0
\end{pmatrix}.
\]
Then $abc = 0$, while
\[
acb=
\begin{pmatrix}
0 & 1\\
0 & 0
\end{pmatrix}\ne 0.
\]
Consequently, $R$ fails to be symmetric.
\end{example}

\begin{question}\label{Q3.03}
Does Condition (3) imply Condition (4)?
\end{question}

From Remark \ref{rem2.09}, Remark \ref{rem2.10} and Proposition \ref{pro2.13},
we readily obtain the following result.

\begin{corollary}\label{cor3.04}
Let $R$ be a $\ast$-ring.
If $abc\in P(R)$ implies $acb \in P(R)$, then $acb=abc$ for all $a,b,c\in R$.
\end{corollary}

Herein, we introduce a classes of like-symmetric rings, whose definition is given below.

\begin{definition}
$R$ is said to be \textit{pro-symmetric} if $abc\in P(R)$ implies $acb \in P(R)$ for all $a,b,c\in R$.
\end{definition}

In the preceding discussion, we have shown that
every pro-symmetric ring is symmetric,
but the converse does not hold in general.
Recall that a ring $R$ is symmetric provided that
for any $x_{1},x_{2},x_{3}\in R$, $x_{1}x_{2}x_{3}=0$ implies
$x_{\sigma(1)}x_{\sigma(2)}x_{\sigma(3)}=0$ for every permutation $\sigma\in S_{3}$,
where $S_3$ denotes the symmetric group on three elements.
Accordingly, we establish the following result for pro-symmetric rings.

\begin{lemma}\label{lem3.06}
Let $R$ be a $\ast$-ring, and let $a, b \in R$. The following statements are equivalent:

(1) $R$ is reversible.

(2) $ab \in P(R)$ implies $ab = ba$.
\end{lemma}

\begin{proof}
$(1)\Rightarrow(2)$. Assume $ab\in P(R)$.
Then $(1-ab)ab=0$. Since $R$ is reversible, we have $b(1-ab)a=0$,
which gives $ba\in E(R)$.
Every reversible ring is abelian, so all idempotent elements are central in $R$.
Hence, we obtain
\[
ba = b(ab)a = ab(ba) = a(ba)b = ab.
\]

$(2)\Rightarrow(1)$. If $ab=0$, then $ab\in P(R)$. By condition (2), $ba=ab=0$,
which proves that $R$ is reversible.
\end{proof}

\begin{proposition}\label{pro3.07}
Let $R$ be a pro-symmetric ring. For any $x_{1},x_{2},x_{3}\in R$, if $x_{1}x_{2}x_{3}\in P(R)$, then
\[
x_{1}x_{2}x_{3}=x_{\sigma(1)}x_{\sigma(2)}x_{\sigma(3)}
\]
for every permutation $\sigma\in S_3$.
\end{proposition}

\begin{proof}
By Corollary \ref{cor3.04}, we immediately have $x_{1}x_{2}x_{3}=x_{1}x_{3}x_{2}$.
Note that every pro-symmetric ring is symmetric and hence reversible.
Since $x_{1}x_{2}x_{3}\in P(R)$ and $R$ is reversible,
Lemma \ref{lem3.06} implies $x_{1}x_{2}x_{3}=x_{2}x_{3}x_{1}\in P(R)$.
Applying Corollary \ref{cor3.04} once more, we get $x_{2}x_{3}x_{1}=x_{2}x_{1}x_{3}$.
Repeating the above reasoning, we further obtain $x_{2}x_{1}x_{3}=x_{3}x_{2}x_{1}=x_{3}x_{1}x_{2}$.
Thus all permuted triple products coincide, which completes the proof.
\end{proof}

Herein, we present two distinct observations.

\begin{remark}\label{rem3.08}
Let $R$ be a $\ast$-ring. The following statements are equivalent for all $a,b,c\in R$:
\begin{enumerate}
\item[(1)] $R$ is reversible.
\item[(2)] $abc \in P(R)$ implies $abc = bca$.
\item[(3)] $abc \in P(R)$ implies $abc = cab$.
\end{enumerate}
Indeed, $(1)\Rightarrow(2)$. It suffices to regard $bc$ as a single element.
Then the desired conclusion follows from Lemma \ref{lem3.06}.
The proof of $(1)\Rightarrow(3)$ is similar.

$(2)\Rightarrow(1)$. If $xy\in P(R)$, write $xy=xy1$. By condition (2), we obtain $xy=yx$.
Hence $R$ is reversible by Lemma \ref{lem3.06}.

The proof of $(3)\Rightarrow(1)$ is similar. If $xy\in P(R)$, write $xy=x1y$.
Then by condition (3), we obtain $xy=yx$.
\end{remark}

\begin{remark}\label{rem3.09}
Let $R$ be a $\ast$-ring. The following statements are equivalent for all $a,b,c\in R$:
\begin{enumerate}
\item[(1)] $R$ is pro-symmetric.
\item[(2)] $abc \in P(R)$ implies $abc = bac$.
\item[(3)] $abc \in P(R)$ implies $abc = cba$.
\end{enumerate}
Indeed, $(1)\Rightarrow(2)$ and $(1)\Rightarrow(3)$. These implications are proved in Proposition \ref{pro3.07}.

$(2)\Rightarrow(1)$. We claim that $R$ is reversible. Indeed,
if $xy\in P(R)$, write $xy=xy1$. By condition (2), we obtain $xy=yx$.
Hence $R$ is reversible by Lemma \ref{lem3.06}.
Now suppose that $abc \in P(R)$. Then by condition (2), we have $bac=abc\in P(R)$.
Again by Lemma \ref{lem3.06}, we obtain $(ac)b = b(ac)\in P(R)$.
Therefore $R$ is pro-symmetric.

The proof of $(3)\Rightarrow(1)$ is similar.
\end{remark}

Recall that a ring $R$ is called $\ast$-symmetric \cite{WFW1} if $abc=0$ implies $acb^{\ast}=0$ for all $a,b,c\in R$.
It has been verified that every pro-symmetric ring is symmetric, and that every $\ast$-symmetric ring is also symmetric.
At the end of this section, we provide several examples to clarify the relationship between pro-symmetric rings and $\ast$-symmetric rings.

\begin{example}
There exists a pro-symmetric ring which is not $\ast$-symmetric.
Let $R=\mathbb{Z}[x]/(x^{2}+x)$.
We know that $R$ is a commutative ring.
Hence $R$ is pro-symmetric.
Nevertheless, $R$ is not $\ast$-symmetric (see \cite{WFW1}, Example 2.1).
\end{example}

\begin{example}
There exists a $\ast$-symmetric ring that is not pro-symmetric.
Let $\mathbb{H}$ be the real quaternion division ring.
As observed in Remark \ref{rem2.09}, we can equip $\mathbb{H}$ with a suitable involution $\ast$.
Take $a=\mathbf{i}$, $b=\mathbf{j}$, and $c=-\mathbf{k}$.
A direct computation gives $abc=1$ and $acb=-1$.
Clearly, $abc=1\in P(\mathbb{H})$, while $acb=-1$ is not idempotent and hence not a projection.
Consequently, $\mathbb{H}$ is not pro-symmetric.

On the other hand, $\mathbb{H}$ is $\ast$-symmetric.
Note that every nonzero element of $\mathbb{H}$ is invertible.
If $xyz=0$ for $x,y,z\in\mathbb{H}$, then at least one of $x,y,z$ must be zero,
which forces $xzy^{\ast}=0$.
Therefore, $\mathbb{H}$ is a $\ast$-symmetric ring.
\end{example}

\section*{Acknowledge}
The authors sincerely thank Prof. Junchao Wei for his valuable comments and suggestions,
which greatly improved the presentation of this paper.
This work is supported by the National Natural Science Foundation of China (12471133)
and the Anhui Provincial Department of Education Natural Science Research Project (2025AHGXZK30855).

\end{document}